\documentclass[a4paper,american,reqno]{amsart}

\def\MyAMSFlag{1}

\usepackage{babel}
\usepackage[T1]{fontenc}
\usepackage[utf8]{inputenc}
\usepackage{csquotes}
\usepackage{siunitx}
\usepackage[dvipsnames,table]{xcolor}
\usepackage{amsmath}
\usepackage{amsfonts}
\usepackage{amssymb}
\usepackage{mathtools}
\usepackage{booktabs}
\usepackage{multirow}
\usepackage{enumitem}
\usepackage{mathdots}
\setlist[enumerate,1]{label=(\roman*)}

\if\MyAMSFlag1
\usepackage[style=numeric,
            giveninits=true,
            maxbibnames=100,
            maxcitenames=2,
            isbn=false,
            doi=false,
            url=false,
            isbn=false,
            backend=biber]{biblatex}
\newcommand{\citet}[1]{\citeauthor{#1}~\cite{#1}}

\fi
\usepackage[colorlinks,
            citecolor=blue,
            urlcolor=blue,
            linkcolor=blue]{hyperref} % should always be the last package

\if\MyAMSFlag1
\newtheorem{theorem}{Theorem}

\newtheorem{lemma}[theorem]{Lemma}

\newtheorem{proposition}[theorem]{Proposition}

\fi

\newcommand{\DecisionProblemPrefix}[1]{}

\newcommand{\MacroColor}[1]{#1}

\newcommand{\Integers}{\mathbb{Z}}

\newcommand{\NonNegativeIntegers}{\mathbb{N}}
\newcommand{\PositiveIntegers}{\mathbb{N}_{+}}
\newcommand{\BinarySet}{\mathbb{B}}
\newcommand{\DomainPlaceholder}{\mathbb{D}}

\newcommand{\ZeroVector}{\mathbf{0}}
\newcommand{\OneVector}{\mathbf{1}}

\newcommand{\ZeroMatrix}{O}

\DeclareMathOperator*{\Argmin}{\MacroColor{argmin}}
\DeclareMathOperator*{\Argmax}{\MacroColor{argmax}}

\newcommand{\OptValFunc}{\MacroColor{V}}
\newcommand{\OptSolSet}{\MacroColor{\mathcal{S}}}
\newcommand{\FeasSolSet}{\MacroColor{\mathcal{F}}}

\newcommand{\BILPConstraintSense}{\ge}

\newcommand{\ComplexityClassWrapper}[1]{\MacroColor{\ensuremath{#1}}}
\newcommand{\ComplexityClassText}[1]{\textup{#1}}

\newcommand{\ClassP}{\ComplexityClassWrapper{\ComplexityClassText{P}}}
\newcommand{\ClassNP}{\ComplexityClassWrapper{\ComplexityClassText{NP}}}

\newcommand{\ClassDP}{\ComplexityClassWrapper{\ComplexityClassText{DP}}}
\newcommand{\ClassBH}[1]{\ComplexityClassWrapper{\ComplexityClassText{BH}_{#1}(\ClassNP{})}}

\newcommand{\ClassSigmaP}[1]{\ComplexityClassWrapper{\Sigma^{p}_{#1}}}
\newcommand{\ClassDeltaP}[1]{\ComplexityClassWrapper{\Delta^{p}_{#1}}}
\newcommand{\ClassThetaP}[1]{\ComplexityClassWrapper{\Theta^{p}_{#1}}}

\newcommand{\ComplexityProblemWrapper}[1]{\MacroColor{\ensuremath{#1}}}
\newcommand{\ComplexityProblemText}[1]{\textup{#1}}

\newcommand{\DecisionProblemBilevelIntegerLP}{\ComplexityProblemWrapper{\ComplexityProblemText{DEC}({\Integers})}}
\newcommand{\DecisionProblemBilevelNonNegativeIntegerLP}{\ComplexityProblemWrapper{\ComplexityProblemText{DEC}({\NonNegativeIntegers})}}
\newcommand{\DecisionProblemBilevelBinaryLP}{\ComplexityProblemWrapper{\ComplexityProblemText{DEC}({\BinarySet})}}
\newcommand{\DecisionProblemBilevelLPWithDomainPlaceholder}{\ComplexityProblemWrapper{\ComplexityProblemText{DEC}({\DomainPlaceholder})}}

\newcommand{\DecisionProblemFeasibilityOfBilevelIntegerLP}{\ComplexityProblemWrapper{\ComplexityProblemText{FEAS}({\Integers})}}
\newcommand{\DecisionProblemFeasibilityOfBilevelNonNegativeIntegerLP}{\ComplexityProblemWrapper{\ComplexityProblemText{FEAS}({\NonNegativeIntegers})}}
\newcommand{\DecisionProblemFeasibilityWithDomainPlaceholder}{\ComplexityProblemWrapper{\ComplexityProblemText{FEAS}({\DomainPlaceholder})}}

\newcommand{\DecisionProblemPricingBilevelIntegerLP}{\ComplexityProblemWrapper{\ComplexityProblemText{DEC}'({\Integers})}}
\newcommand{\DecisionProblemPricingBilevelNonNegativeIntegerLP}{\ComplexityProblemWrapper{\ComplexityProblemText{DEC}'({\NonNegativeIntegers})}}
\newcommand{\DecisionProblemPricingBilevelBinaryLP}{\ComplexityProblemWrapper{\ComplexityProblemText{DEC}'({\BinarySet})}}
\newcommand{\DecisionProblemPricingBilevelLPWithDomainPlaceholder}{\ComplexityProblemWrapper{\ComplexityProblemText{DEC}'({\DomainPlaceholder})}}

\newcommand{\EqTagDecisionProblemOnKLP}[2]{\ComplexityProblemWrapper{#1\ComplexityProblemText{-#2}}}

\newcommand{\DecisionProblemFeas}[1]{\hyperref[LabelDecisionProblemFeas]{\EqTagDecisionProblemOnKLP{#1}{FEAS}}}
\newcommand{\DecisionProblemVal}[1]{\hyperref[LabelDecisionProblemVal]{\EqTagDecisionProblemOnKLP{#1}{VAL}}}
\newcommand{\DecisionProblemUnb}[1]{\hyperref[LabelDecisionProblemUnb]{\EqTagDecisionProblemOnKLP{#1}{UNB}}}
\newcommand{\DecisionProblemRecognize}[1]{\hyperref[LabelDecisionProblemRecognize]{\EqTagDecisionProblemOnKLP{#1}{RECOG}}}
\newcommand{\DecisionProblemAttain}[1]{\hyperref[LabelDecisionProblemAttain]{\EqTagDecisionProblemOnKLP{#1}{ATAIN}}}
\newcommand{\SearchProblemObj}[1]{\hyperref[LabelSearchProblemObj]{\EqTagDecisionProblemOnKLP{#1}{OBJ}}}
\newcommand{\SearchProblemSol}[1]{\hyperref[LabelSearchProblemSol]{\EqTagDecisionProblemOnKLP{#1}{SOL}}}

\newcommand{\NVariables}{\MacroColor{n}}
\newcommand{\NConstraints}{\MacroColor{m}}

\newcommand{\DecisionVersionThreshold}{\MacroColor{W}}

\newcommand{\SATTruthAssignment}{\MacroColor{\nu}}

\newcommand{\EqTagGeneralKLP}[3]{\MacroColor{\ensuremath{\mathrm{LP}_{\ensuremath{#1}}^{\ensuremath{#2}}(\ensuremath{#3})}}}
\newcommand{\EqrefGeneralKLP}[3]{\textup{(}\text{\hyperref[LabelGeneralKLP]{$\EqTagGeneralKLP{#1}{#2}{#3}$}\textup{)}}}

\newcommand{\EqTagJeroslowPenaltyBefore}[3]{\MacroColor{\ensuremath{\mathrm{J}_{\ensuremath{#1}}^{\ensuremath{#2}}(\ensuremath{#3})}}}
\newcommand{\EqrefJeroslowPenaltyBefore}[3]{\textup{(}\text{\hyperref[LabelJeroslowPenaltyBefore]{$\EqTagJeroslowPenaltyBefore{#1}{#2}{#3}$}\textup{)}}}

\newcommand{\EqTagJeroslowPenaltyAfter}[3]{\MacroColor{\ensuremath{\overline{\mathrm{J}}_{\ensuremath{#1}}^{\ensuremath{#2}}(\ensuremath{#3})}}}
\newcommand{\EqrefJeroslowPenaltyAfter}[3]{\textup{(}\text{\hyperref[LabelJeroslowPenaltyAfter]{$\EqTagJeroslowPenaltyAfter{#1}{#2}{#3}$}\textup{)}}}

\newcommand{\EqTagQThreeSATReducedToKLPLeadersProblem}[3]{\MacroColor{\ensuremath{\textup{QLP}_{\ensuremath{#1}}^{\ensuremath{#2}}\ensuremath{#3}}}}
\newcommand{\EqTagQThreeSATReducedToKLPFollowersProblem}[3]{\MacroColor{\ensuremath{\textup{QLP}_{\ensuremath{#1}}^{\ensuremath{#2}}\ensuremath{#3}}}}

\newcommand{\EqRefQThreeSATReducedToKLPLeadersProblem}[3]{\textup{(}\text{\hyperref[eq:LabelQThreeSATReducedToKLPLeadersProblem]{$\EqTagQThreeSATReducedToKLPLeadersProblem{#1}{#2}{#3}$}\textup{)}}}
\newcommand{\EqRefQThreeSATReducedToKLPFollowersProblem}[3]{\textup{(}\text{\hyperref[eq:LabelQThreeSATReducedToKLPLeadersProblem]{$\EqTagQThreeSATReducedToKLPFollowersProblem{#1}{#2}{#3}$}\textup{)}}}

\newcommand{\NewEmph}[1]{\textbf{#1}}

\newcommand{\BILPIntegers}{\textup{BILP}($\Integers$)}
\newcommand{\BILPNonNegativeIntegers}{\textup{BILP}($\NonNegativeIntegers$)}
\newcommand{\BILPBinarySet}{\textup{BILP}($\BinarySet$)}
\newcommand{\BILPWithDomainPlaceholder}{\textup{BILP(}$\DomainPlaceholder$\textup{)}}

\newcommand{\BILPSetOfBooleanFormulae}{\mathcal{B}}
\newcommand{\BILPBooleanFormula}{H}
\newcommand{\BILPQFreeBooleanFormula}{F}
\newcommand{\BILPSetOfThreeCNF}{\mathcal{F}}
\newcommand{\BILPSATNVar}{r}
\newcommand{\BILPSATNClause}{q}

\bibliography{references}

\title[On the Complexity of Bilevel Integer Linear Programming]{On the Computational Complexity of \\ Bilevel Integer Linear Programming}
\author{Nagisa Sugishita}
\author{Margarida Carvalho}
\address[Nagisa Sugishita, Margarida Carvalho]{André-Aisenstadt Pavillon, 2920 Tour Road, Montreal, Quebec H3T 1N8, Canada}
\email{nagisa.sugishita@hec.ca}

\date{\today}

\begin{document}

\begin{abstract}
We investigate the computational complexity of bilevel integer linear programming. 
While Jeroslow~(1985) established that the decision version of this problem is $\Sigma^p_2$-complete when restricted to binary variables, we prove that this $\Sigma^p_2$-completeness persists even for general integer variables, settling a question that remained open for over 40 years.
Furthermore, we analyze the impact of various structural assumptions on computational complexity. 
Notably, we strengthen the result of Köppe et al.~(2010) by proving polynomial-time solvability whenever the total number of upper- and lower-level variables is fixed, without any additional assumptions.
\end{abstract}

\maketitle

\section{Introduction}\label{sec:introduction}

Bilevel programming is a hierarchical optimization problem in which the upper-level player makes the decision first, after which the lower-level player responds.
Due to its modeling flexibility, bilevel programming has attracted considerable attention in the literature. We refer the reader to \citet{carvalho2025integer} and \citet{beck2026bilevel} for recent surveys.
In this paper, we consider the optimistic formulation, in which, if the lower-level problem admits multiple optimal solutions, the one most favorable to the upper-level player is selected.

The most fundamental variant, bilevel linear programming (LP), has been extensively studied.
\citet{Jeroslow1985} showed that its decision version is \ClassNP{}-hard, while \citet{Buchheim2023} proved that it belongs to \ClassNP{}, establishing \ClassNP{}-completeness.
\citet{Sugishita2026SingleUpperLevelVariable} showed that the problem remains \ClassNP{}-complete even when there is only a single upper-level variable and no upper-level constraints.
In contrast, fixing the number of lower-level variables or lower-level constraints renders the problem tractable, as summarized by \citet{ketkov2025note}.

The computational complexity of bilevel binary LP appears higher than that of bilevel LP due to the binary constraints.
\citet{Jeroslow1985} showed that the decision version of bilevel binary LP is \ClassSigmaP{2}-complete, and Sugishita and Carvalho~\cite{Sugishita2026,sugishita2026price} extended the result to the bilevel mixed binary LP case.
Moreover, several bilevel knapsack problems have been shown to be \ClassSigmaP{2}-complete, some of which imply the \ClassSigmaP{2}-completeness of bilevel binary LP with a single lower-level constraint~\cite{caprara2014study}.
Interdiction problems have also been extensively studied in this context~\cite{nabli2022complexity,boggio2025completeness,grune2026complexity}.

Less is known about the computational complexity of bilevel integer LP.
On the membership side, we prove that the decision version of bilevel integer LP is in \ClassSigmaP{2}. 
Together with the hardness result of \citet{Jeroslow1985}, this implies its \ClassSigmaP{2}-completeness.
We also examine the problem under various structural assumptions.
For example, the problem belongs to \ClassNP{} if the number of lower-level variables is fixed.
However, on the hardness side, we show that the problem remains \ClassSigmaP{2}-complete even when the number of upper-level variables is fixed.
We note that \citet{koppe2010parametric} proved polynomial-time solvability when the total number of upper- and lower-level variables is fixed and both feasible sets are bounded.
We strengthen this result by showing that polynomial-time solvability still holds without the boundedness assumption.
In this paper, we also establish, as an auxiliary result, the hardness of a minimax knapsack problem with a single shared capacity constraint.
Table~\ref{tab:summary} summarizes some of the main results of this paper.
Finally, to demonstrate the flexibility of our analysis, we study a pricing variant of bilevel integer LP.
In particular, we show the \ClassSigmaP{2}-completeness of its decision version.

\begin{table}
\caption{%
Computational complexity of the decision versions of bilevel LP with integer, nonnegative integer, binary, and continuous variables, alongside special cases where the numbers of upper-level variables $n_1$, lower-level variables $n_2$, or lower-level constraints $m_2$ are fixed to some positive constant $k$.
The results in bold font are established in this paper.}
\label{tab:summary}
\small
\setlength{\tabcolsep}{3pt}
\begin{tabular}{ccccccc}
% & & & & Upper-level & Lower-level & Lower-level \\
Domain & & General & $n_1 = k$ & $n_2 = k$ & $m_2 = k$ & $n_1 + n_2 = k$ \\
\midrule
$\Integers$ & & \NewEmph{\ClassSigmaP{2}-comp.} & \NewEmph{\ClassSigmaP{2}-comp.} & \NewEmph{\ClassNP{}-comp.} & \NewEmph{\ClassNP{}-comp.} & \NewEmph{\ClassP{}} \\%\cite{koppe2010parametric} \\
$\NonNegativeIntegers$ & & \NewEmph{\ClassSigmaP{2}-comp.} & \NewEmph{\ClassSigmaP{2}-comp.} & \NewEmph{\ClassNP{}-comp.} & \NewEmph{\ClassSigmaP{2}-comp.} & \NewEmph{\ClassP{}} \\%\cite{koppe2010parametric} \\
$\BinarySet$ & & \ClassSigmaP{2}-comp. \cite{Jeroslow1985} & \NewEmph{\ClassDeltaP{2}-comp.} & \NewEmph{\ClassNP{}-comp.} & \NewEmph{\ClassSigmaP{2}-comp.} & \ClassP{} \cite{koppe2010parametric} \\
$\mathbb{R}$ 
% \cite{Buchheim2023,Jeroslow1985,ketkov2025note,Sugishita2026SingleUpperLevelVariable} 
& & \ClassNP{}-comp. \cite{Jeroslow1985,Buchheim2023} & \ClassNP{}-comp. \cite{Sugishita2026SingleUpperLevelVariable} & \ClassP{} \cite{ketkov2025note} & \ClassP{} \cite{ketkov2025note} & P \cite{ketkov2025note} \\
\bottomrule
\end{tabular}
\end{table}

\textbf{Paper Structure.}
This paper is organized as follows.
In Section~\ref{sec:Bilevel}, we introduce the notations and problem setup.
We also establish the \ClassSigmaP{2}-completeness of bilevel integer LP.
Section~\ref{sec:FixedDimensions} concerns the fixed-parameter cases.
Section~\ref{sec:Pricing} investigates the pricing variants and extends our results to them.

\textbf{Conventions.}
Given an optimization instance $P$, we use $\OptValFunc(P)$,
% \MargaridaSideComment{Maybe we should change the notation $v(p)$ to $V(P)$ because we use $v$ a lot as a binary variable.}
$\FeasSolSet(P)$, and $\OptSolSet(P)$ to denote the optimal objective value, the set of feasible solutions, and the set of optimal solutions, respectively.
We write $\Integers = \{0, \pm 1, \pm 2, \ldots\}$, $\NonNegativeIntegers{} = \{0, 1, \ldots\}$, $\PositiveIntegers{} = \{1, 2, \ldots\}$, and $\BinarySet = \{0, 1\}$.
We use $\ZeroVector$ (respectively, $\OneVector$) to denote the all-zeros (respectively, all-ones) vector, with the dimension being clear from the context.
Similarly, we use $\ZeroMatrix$ to denote the zero matrix of an appropriate size.
We use subscripts to denote elements of a vector: for a vector $v$, the $i$th element is denoted by $v_i$.
To reduce clutter, we sometimes refer to column vectors inline;
given $u \in \mathbb{R}^{d_1}$ and $v \in \mathbb{R}^{d_2}$, we use
the notation $(u, v)$ to denote the column vector
$
\begin{pmatrix} u \\ v \end{pmatrix} \in \mathbb{R}^{d_1 + d_2}.
$

\section{Bilevel Integer Linear Programming}
\label{sec:Bilevel}

In this section, we introduce the notations and problem setup.
We then prove that the decision version of bilevel integer LP is \ClassSigmaP{2}-complete.

\subsection{Notation and Problem Definition}
\label{sec:Bilevel:NotationAndProblemDefinition}

In this paper, we use \BILPIntegers{}, \BILPNonNegativeIntegers{}, and \BILPBinarySet{} to denote bilevel integer LP with unrestricted integer variables, bilevel integer LP with nonnegative integer variables, and bilevel binary LP, respectively.
We give a precise definition of the problem below.

Let $n_l \in \NonNegativeIntegers{}$ and $m_l \in \NonNegativeIntegers{}$ for each $l \in \{1, 2\}$.
% Let $A_{1 1} \in \Integers^{\NConstraints_1 \times n_1}$, $A_{1 2} \in \Integers^{\NConstraints_1 \times n_2}$, $A_{2 1} \in \Integers^{\NConstraints_2 \times n_1}$, $A_{2 2} \in \Integers^{\NConstraints_2 \times n_2}$, $b_1 \in \Integers^{\NConstraints_1}$, $b_2 \in \Integers^{\NConstraints_2}$, $c_{1 1} \in \Integers^{\NVariables_1}$, $c_{1 2} \in \Integers^{\NVariables_2}$, and $c_{2 2} \in \Integers^{\NVariables_2}$.
For each $l \in \{1, 2\}$ and $i \in \{1, 2\}$, let $A_{l i} \in \Integers^{\NConstraints_l \times \NVariables_i}$, $b_l \in \Integers^{\NConstraints_l}$, and $c_{l i} \in \Integers^{\NVariables_i}$.
We collectively write $A = ( A_{11}, A_{12}, A_{21}, A_{22} )$, $b = ( b_{1}, b_{2} )$, and $c = ( c_{1 1}, c_{1 2}, c_{2 2} )$.
For $\DomainPlaceholder \in \{\Integers, \NonNegativeIntegers, \BinarySet\}$, a \BILPWithDomainPlaceholder{} instance is given by
\begin{align}
\min_{x_1, x_2}
\left\{
c_{1 1}^{\top} x_1 + c_{1 2}^{\top} x_2
:
\begin{array}{l}
A_{1 1} x_1 + A_{1 2} x_2
\ge
b_1,
x_1 \in \DomainPlaceholder^{n_1},
\\
\displaystyle
x_2
\in
\Argmin_{x_2'}
\left\{
c_{2 2}^{\top} x_2'
:
\begin{array}{l}
A_{2 1} x_1 + A_{2 2} x_2'
\ge
b_2,
\\
x_2' \in \DomainPlaceholder^{n_2}
\end{array}
\right\}
\end{array}
\right\}.
% \tag{\ensuremath{\text{\EqTagDefBILP}(A, b, c)}}
\label{Eq:BILP}
\end{align}

We use inequality constraints in~\eqref{Eq:BILP} to maintain a unified notation across \BILPIntegers{}, \BILPNonNegativeIntegers{}, and \BILPBinarySet{}.
In the integer LP literature, it is also common to use an equality-constrained problem with nonnegative integer variables.
All the results regarding \BILPNonNegativeIntegers{} in this paper
%, including those in Section~\ref{sec:FixedDimensions}, 
hold if we use $=$ instead of $\ge$ in~\eqref{Eq:BILP}. 
However, \BILPIntegers{} with equality constraints is polynomially solvable, since the feasible set is a lattice coset and the problem is tractable via Hermite normal form.

An upper-level constraint is termed a linking (coupling) constraint if and only if it involves the lower-level variable $x_2$.
One can show that all non-linking constraints can be ``moved'' to the lower-level problem without altering the optimal solutions~\cite{Sugishita2026}.
In particular, if there are no linking constraints (i.e., $A_{1 2} = \ZeroMatrix$), one can obtain an equivalent instance without upper-level constraints (i.e., $m_1 = 0$).

We can view a single-level instance as a special case of a bilevel instance where the lower-level problem has no variables or constraints (i.e., $\NVariables_2 = \NConstraints_2 = 0$).
A bilevel instance without upper-level variables or constraints (i.e., $\NVariables_1 = \NConstraints_1 = 0$) reduces to the optimization over the optimal solution set of the lower-level problem:
\begin{equation}
\label{Eq:BilevelWithoutUpperLevelVariablesOrConstraints}
\min_{x_2}
\left\{
c_{1 2}^{\top} x_2
:
x_2
\in
\Argmin_{x_2'}
\{
c_{2 2}^{\top} x_2' :
A_{2 2} x_2'
\BILPConstraintSense
 b_2,
x_2' \in \DomainPlaceholder^{\NVariables_2}
\}
\right\}.
\end{equation}
Another special case is a minmax problem (i.e., $c_{1 2} = - c_{2 2}$).
When dealing with a minmax problem lacking upper-level constraints (i.e., $\NConstraints_1 = 0$), we write~\eqref{Eq:BILP} as
\begin{align}
\label{Eq:MinMaxWithoutUpperLevelConstraints}
\min_{x_1 \in \DomainPlaceholder^{\NVariables_1}}
\left\{
c_{1 1}^{\top} x_1 + \max_{x_2}
\left\{
c_{1 2}^{\top} x_2
:
% \begin{array}{l}
A_{2 1} x_1 + A_{2 2} x_2
\ge
b_2,
x_2 \in \DomainPlaceholder^{n_2}
% \end{array}
\right\}
% :
% A_{1 1} x_1 \ge b_1
\right\}.
\end{align}
However, this should be understood as a notational convenience, with the actual definition being~\eqref{Eq:BILP}.
In particular, any $x_1$ for which the lower-level problem is infeasible should be considered infeasible for the bilevel instance.
Furthermore, if the lower-level problem is infeasible for all $x_1 \in \DomainPlaceholder^{\NVariables_1}$, the optimal objective value of the bilevel instance should be understood to be $\infty$.

For each $\DomainPlaceholder \in \{\Integers, \NonNegativeIntegers, \BinarySet\}$, we consider the decision version of \BILPWithDomainPlaceholder{}:
\begin{align}
&
\parbox{0.8\textwidth}{Given $(A, b, c)$ and an integer $\DecisionVersionThreshold$, does there exist a feasible solution for~\eqref{Eq:BILP} whose objective value is less than or equal to $\DecisionVersionThreshold$?}
\tag{\DecisionProblemBilevelLPWithDomainPlaceholder{}}
\label{Eq:BILP-VAL}
\end{align}

To capture the computational complexity precisely, we give particular consideration to instances satisfying the following restrictions.
% \NagisaSideComment{We used (C2) "the coefficients are at most $n$" in the paper for multilevel (continuous) LP. But some of our results (i.e, Prop.~\ref{Prop:FixedLowerLevel}) in this paper need a constant bound, so I chose $1$.}
\begin{enumerate}[label=(C\arabic*)]
\item
\label{Condition1}
There are no linking constraints, i.e., $A_{1 2} = \ZeroMatrix$;
\item
\label{Condition2}
All entries in $A$, $b$, $c$, and $\DecisionVersionThreshold$ are integers between $-1$ and $1$;
\item
\label{Condition3}
It is a minmax problem, i.e., $c_{1 2} = -c_{2 2}$.
\end{enumerate}

\citet{Jeroslow1985} establishes the following result.
\begin{theorem}[\citet{Jeroslow1985}]
\label{theorem:ComplexityOfBilevelBinaryLP}
The decision problem \DecisionProblemBilevelBinaryLP{} is \ClassSigmaP{2}-complete. Moreover, this remains valid even when the input is restricted to satisfy conditions~\ref{Condition1}--\ref{Condition3}.
\end{theorem}

We now present one of our main results.
\begin{theorem}
\label{theorem:ComplexityOfBilevelIntegerLP}
The decision problems \DecisionProblemBilevelIntegerLP{} and \DecisionProblemBilevelNonNegativeIntegerLP{} are \ClassSigmaP{2}-complete.
Moreover, this remains valid even when the input is restricted to satisfy conditions~\ref{Condition1}--\ref{Condition3}.
\end{theorem}

The hardness follows from Theorem~\ref{theorem:ComplexityOfBilevelBinaryLP}.
In this paper, we establish the membership result in Theorem~\ref{theorem:ComplexityOfBilevelIntegerLP}.

\subsection{Reformulations of Boolean Satisfiability Problems}
\label{sec:BILPFormulationsBooleanSatisfiabilityProblems}

In this section, we discuss Boolean satisfiability problems and their formulations as optimization instances.
Let $\BILPSetOfThreeCNF$ be the set of Boolean formulae in 3-conjunctive normal form (3CNF).
It is known that the following is \ClassNP{}-complete \cite{papadimitriou1993computational}:
\begin{equation}
\tag{3SAT}
\label{Eq:3SAT}
\parbox{0.8\textwidth}{Given $\BILPQFreeBooleanFormula \in \BILPSetOfThreeCNF$, is $\BILPQFreeBooleanFormula$ satisfiable?}
\end{equation}
Let \( \BILPQFreeBooleanFormula \in \BILPSetOfThreeCNF \) with \( \BILPSATNVar \) variables \( \SATTruthAssignment{}_{1}, \ldots, \SATTruthAssignment{}_{\BILPSATNVar} \) and \( \BILPSATNClause \) clauses.
Following \cite{marcotte2005bilevel}, \( \BILPQFreeBooleanFormula \) can be represented by a system of linear inequalities over binary variables.
For example,
$
F = (\SATTruthAssignment{}_1 \vee \neg \SATTruthAssignment{}_2) \wedge (\neg \SATTruthAssignment{}_1 \vee \neg \SATTruthAssignment{}_2 \vee \SATTruthAssignment{}_3)
$
corresponds to
\[
u_1 + (1 - u_2) \geq 1, \quad (1 - u_1) + (1 - u_2) + u_3 \geq 1, \quad u \in \BinarySet^3.
\]
A satisfying truth assignment corresponds exactly to a feasible solution, where $\SATTruthAssignment{}_i = \text{true}$ corresponds to $u_i = 1$.
In general, given any \( \BILPQFreeBooleanFormula \), one can construct
$
A_{\BILPQFreeBooleanFormula} u \geq \OneVector - b_{\BILPQFreeBooleanFormula}, \ u \in \BinarySet^{\BILPSATNVar},
$
where \( A_{\BILPQFreeBooleanFormula} \in \Integers^{\BILPSATNClause \times \BILPSATNVar} \) and \( b_{\BILPQFreeBooleanFormula} \in \Integers^{\BILPSATNClause} \) such that (assuming that each Boolean variable appears at most once in each clause) each entry of $A_{\BILPQFreeBooleanFormula}$ is between $-1$ and $1$, and each entry of $b_{\BILPQFreeBooleanFormula}$ is between $0$ and $3$.
In the following, we state that a binary vector $u$ satisfies \( \BILPQFreeBooleanFormula \) if and only if it satisfies $A_{\BILPQFreeBooleanFormula} u \ge \OneVector - b_{\BILPQFreeBooleanFormula}$.
% the corresponding truth assignment satisfies \( \BILPQFreeBooleanFormula \).

One can reformulate~\ref{Eq:3SAT} as an optimization problem
$$
\max_{u, v}
\{
v :
A_{\BILPQFreeBooleanFormula} u \ge v \OneVector - b_{\BILPQFreeBooleanFormula}, \
u \in \BinarySet^{\BILPSATNVar}, v \in \BinarySet
\},
$$
which is $1$ if $\BILPQFreeBooleanFormula$ is satisfiable, and $0$ otherwise.
Although some entries in $b_{\BILPQFreeBooleanFormula}$ are larger than $1$, one can reformulate the instance such that all data lies between $-1$ and $1$ by introducing auxiliary variables $w \in \BinarySet^3$:
\begin{equation}\label{SAT:entresin01}
\max_{u, v, w}
\left\{
v :
% \begin{array}{l}
A_{\BILPQFreeBooleanFormula} u \ge v \OneVector - B_{\BILPQFreeBooleanFormula} w,
% \\
\
% w_{1} = w_{2} = w_{3} = 1,
w = \OneVector,
% \\
\
u \in \BinarySet^{\BILPSATNVar}, v \in \BinarySet, w \in \BinarySet^3
% \end{array}
\right\},
\end{equation}
where each row of $B_{\BILPQFreeBooleanFormula}$ is the unary representation of the corresponding row of $b_{\BILPQFreeBooleanFormula}$.
% \Margarida{I would just say that $B_F$ corresponds to the unary writing of $b_F$. In my opinion, this would make the equivalence much more obvious because it gives the key idea; otherwise, the reader needs to figure out that this is just the unary encoding. I had to spend a few minutes to understand this myself.}
% We note that by replacing $v \in \BinarySet$ with $v \in \BinarySet^{\BILPSATNClause}$, one can compute the maximum number of simultaneously satisfiable clauses:
% $$
% \max_{u, v, w}
% \left\{
% \sum_{i = 1}^{\BILPSATNClause} v_{i} :
% \begin{array}{l}
% A_{\BILPQFreeBooleanFormula} u \ge v - B_{\BILPQFreeBooleanFormula} w,
% \\
% w = \OneVector,
% \\
% u \in \BinarySet^{\BILPSATNVar}, v \in \BinarySet^{\BILPSATNClause}, w \in \BinarySet^3
% \end{array}
% \right\}.
% $$

Next, we consider a quantified version of the satisfiability problem.
Let $\BILPSetOfBooleanFormulae$ denote the set of closed quantified Boolean formulae $\BILPBooleanFormula$ in prenex normal form of the form
\begin{equation}
\exists s \forall t [\BILPQFreeBooleanFormula(s, t) = 0],
\label{Eq:QSATH}
\end{equation}
where $s$ and $t$ are disjoint sets of Boolean variables and $\BILPQFreeBooleanFormula \in \BILPSetOfThreeCNF$.
% \Margarida{We already used $s$ for the number of clauses in a 3SAT instance. Maybe we can use $q$ for number of clauses?}
We assume $s$ and $t$ have the same number of Boolean variables (this can be achieved by padding extra variables), and denote it by $\BILPSATNVar$.
The following decision problem is \ClassSigmaP{2}-complete~\cite{wrathall1976complete}:
\begin{equation}
% \tag{Q3SAT}
\notag
\parbox{0.8\textwidth}{%
Given $\BILPBooleanFormula \in \BILPSetOfBooleanFormulae$, is $\BILPBooleanFormula$ true?
}
\end{equation}
A straightforward extension of the above construction gives a transformation of $\BILPBooleanFormula \in \BILPSetOfBooleanFormulae$ of form~\eqref{Eq:QSATH} to a \BILPBinarySet{} instance
\begin{equation}
\label{Eq:Jeroslow}
% \tag{\ensuremath{\text{\EqTagQSATBILP}(\BILPBooleanFormula)}}
\min_{s \in \BinarySet^{\BILPSATNVar}}
\max_{t, v}
\left\{
v :
% \begin{array}{l}
% C_{\BILPBooleanFormula} s
% +
% A_{\BILPBooleanFormula} t
A_{\BILPQFreeBooleanFormula}
\begin{pmatrix}
s \\ t
\end{pmatrix}
\ge
v \OneVector
- 
% b_{\BILPBooleanFormula},
b_{\BILPQFreeBooleanFormula},
% \\
\
t \in \BinarySet^{\BILPSATNVar}, v \in \BinarySet
% \end{array}
\right\},
\end{equation}
% \MargaridaSideComment{I think it should be $A_F \begin{pmatrix}
%     s\\
%     t
% \end{pmatrix}\geq v 1 -b_F$. Otherwise, I guess we need to define the notation above... }%
which is $0$ if $\BILPBooleanFormula$ is true and $1$ otherwise.
With auxiliary variables as in \eqref{SAT:entresin01}, this instance can be transformed to satisfy conditions~\ref{Condition1}--\ref{Condition3}.
Using a similar line of argument \citet{Jeroslow1985} showed that the decision problem \DecisionProblemBilevelBinaryLP{} is \ClassSigmaP{2}-hard even under conditions~\ref{Condition1}--\ref{Condition3}.

\subsection{Estimate of Optimal Objective Value}
\label{sec:Bilevel:EstimateOfOptimalObjectiveValue}

In this section, we establish an estimate of the encoding size of an optimal solution.

\begin{proposition}
\label{prop:EstimateOfOptimalObjectiveValue}
There exists a polynomial $\psi$ such that the following holds:
\begin{enumerate}
\item
For any $\DomainPlaceholder \in \{\Integers, \NonNegativeIntegers\}$ and \BILPWithDomainPlaceholder{} instance, if it is feasible and not unbounded, there exists an optimal solution with an encoding size of at most $\psi(\sigma)$, where $\sigma$ is the encoding size of the instance;
\item
For any $\DomainPlaceholder \in \{\Integers, \NonNegativeIntegers\}$ and \BILPWithDomainPlaceholder{} instance, if it is feasible and not unbounded, its optimal objective value has an encoding size of at most $\psi(\sigma)$, where $\sigma$ is the encoding size of the instance.
\end{enumerate}
\end{proposition}

Note that Proposition~\ref{prop:EstimateOfOptimalObjectiveValue} implies an analogous result for a feasible solution.
We use the following lemma for the proof of Proposition~\ref{prop:EstimateOfOptimalObjectiveValue}.

\begin{lemma}[E.g., \citet{weismantel1998test}]
\label{lemma:TestSetOfILP}
Let $n, m \in \PositiveIntegers$, $A \in \Integers^{m \times n}$ and $c \in \Integers^{n}$.
Then, there exists a finite set $T \subset \Integers^{n}$ such that the following holds:
\begin{enumerate}
\item
Each $t \in T$ has an encoding size that is polynomial in the encoding sizes of $A$ and $c$ 
% \MargaridaSideComment{should we add ''in $A,b$'' (although I agree that only a sleepy reader would have doubts about it)}%
and satisfies $A t = 0$, and $c^{\top} t < 0$;
\item
For any $b \in \Integers^{m}$ and $x \in \NonNegativeIntegers^{n}$, $x$ is optimal for $\min_{x} \{ c^{\top} x : A x = b, x \in \NonNegativeIntegers^{n} \}$ if and only if $x$ is feasible and $x + t \not\ge \ZeroVector$ for all $t \in T$.
\end{enumerate}
\end{lemma}

\begin{proof}[Proof of Proposition~\ref{prop:EstimateOfOptimalObjectiveValue}]
% Let $(A, b, c)$ be a bilevel integer LP instance~\eqref{Eq:BILP}.
By adding slack variables and splitting free variables if necessary, rewrite~\eqref{Eq:BILP} as
\begin{align}
\min_{\bar{x}_1, \bar{x}_2}
\left\{
\bar{c}_{1 1}^{\top} \bar{x}_1 + \bar{c}_{1 2}^{\top} \bar{x}_2
:
\begin{array}{l}
\bar{A}_{1 1} \bar{x}_1 + \bar{A}_{1 2} \bar{x}_2 = \bar{b}_1,
\bar{x}_1 \in \NonNegativeIntegers^{\bar{n}_1},
\\
\displaystyle
\bar{x}_2
\in
\Argmin_{\bar{x}_2'}
\left\{
\bar{c}_{2 2}^{\top} \bar{x}_2'
:
\begin{array}{l}
\bar{A}_{2 1} \bar{x}_1 + \bar{A}_{2 2} \bar{x}_2'
=
\bar{b}_2,
\\
\bar{x}_2' \in \NonNegativeIntegers^{\bar{n}_2}
\end{array}
\right\}
\end{array}
\right\}.
% \tag{\ensuremath{\overline{\text{\EqTagDefBILP}}}}
\label{Eq:BILP-Equality-Form}
\end{align}
Given a feasible solution for the original instance~\eqref{Eq:BILP}, one can construct a feasible solution for~\eqref{Eq:BILP-Equality-Form} efficiently, and vice versa.

\newcommand{\Pq}{P_{q_1, \ldots, q_I}}
Let $T := \{ t^{l} : l = 1, \ldots, I \}$ be the test set for the lower-level problem.
For any $\bar{x}_1$, $\bar{x}_2$ is optimal for the lower-level problem if and only if it is feasible and $\bar{x}_2 + t^{l}$ is not feasible for all $l = 1, \ldots, I$.
Thus, the feasible set of \eqref{Eq:BILP-Equality-Form} can be rewritten as
\begin{align*}
\FeasSolSet\eqref{Eq:BILP-Equality-Form}
&=
\left\{
\begin{pmatrix}
\bar{x}_1 \\ \bar{x}_2
\end{pmatrix}
\in
\NonNegativeIntegers^{\bar{n}_1 + \bar{n}_2}
:
\begin{array}{l}
\bar{A}_{1 1} \bar{x}_1 + \bar{A}_{1 2} \bar{x}_2 = \bar{b}_1,
\\
\bar{A}_{2 1} \bar{x}_1 + \bar{A}_{2 2} \bar{x}_2 = \bar{b}_2,
\\
\bar{x}_2 + t^{l} \not\ge 0,
\qquad
\forall l = 1, \ldots, I
\end{array}
\right\}.
\end{align*}
Thus, we have
$
\FeasSolSet\eqref{Eq:BILP-Equality-Form}
=
\cup_{(q_1, \ldots, q_I) \in J} \Pq,
$
where $J := \{ 1, \ldots, \bar{\NVariables}_2 \}^{I}$ and
$$
\Pq
:=
\left\{
\begin{pmatrix}
\bar{x}_1 \\ \bar{x}_2
\end{pmatrix}
\in
\NonNegativeIntegers^{\bar{n}_1 + \bar{n}_2}
:
\begin{array}{l}
\bar{A}_{1 1} \bar{x}_1 + \bar{A}_{1 2} \bar{x}_2 = \bar{b}_1,
\\
\bar{A}_{2 1} \bar{x}_1 + \bar{A}_{2 2} \bar{x}_2 = \bar{b}_2,
\\
(\bar{x}_2 + t^{l})_{q_l} \le -1,
\qquad
\forall l = 1, \ldots, I
\end{array}
\right\}
$$
for each $(q_1, \ldots, q_I) \in J$.
Therefore,
$$
\OptValFunc\eqref{Eq:BILP-Equality-Form}
= \min_{(q_1, \ldots, q_I) \in J} \min_{\bar{x}_1, \bar{x}_2} \{
\bar{c}_{1 1}^{\top} \bar{x}_1 + \bar{c}_{1 2}^{\top} \bar{x}_2
:
(\bar{x}_1, \bar{x}_2) \in \Pq
\}
$$
and if $\OptSolSet\eqref{Eq:BILP-Equality-Form} \ne \emptyset$, there exists $(q_1, \ldots, q_I) \in J$ such that
\begin{equation}
\label{Eq:SolSetUnionForm}
\Argmin_{\bar{x}_1, \bar{x}_2} \{
\bar{c}_{1 1}^{\top} \bar{x}_1 + \bar{c}_{1 2}^{\top} \bar{x}_2
:
(\bar{x}_1, \bar{x}_2) \in \Pq
\}
\subseteq \OptSolSet\eqref{Eq:BILP-Equality-Form}.
\end{equation}
For each $(q_1, \ldots, q_I) \in J$, $\Pq$ is defined by exponentially many constraints.
However, its dimension is polynomial, and the coefficients of each constraint have polynomial encoding size. 
This guarantees that, whenever the optimization problem on the left-hand side of~\eqref{Eq:SolSetUnionForm} admits an optimal solution, it admits one of polynomial encoding size~\cite{Schrijver1998}.
As this bound depends only on the dimension and coefficient magnitudes, which is common to every $\Pq$, it holds uniformly over all $\bar{\NVariables}_2^I$ polyhedra, so their minimum $\OptValFunc\eqref{Eq:BILP-Equality-Form}$ inherits it.
Now, the claim follows.
\end{proof}

\subsection{Proof of Theorem~\ref{theorem:ComplexityOfBilevelIntegerLP}}
\label{sec:Bilevel:ProofOfMainTheorem}

We give the proof of Theorem~\ref{theorem:ComplexityOfBilevelIntegerLP} in this section.
We begin with the study of the decision of feasibility.
For $\DomainPlaceholder \in \{\Integers, \NonNegativeIntegers\}$, consider
\begin{align}
&
\parbox{0.8\textwidth}{Given $(A, b, c)$, does there exist a feasible solution for~\eqref{Eq:BILP}?}
\tag{\DecisionProblemFeasibilityWithDomainPlaceholder{}}
\label{Eq:BILP-FEAS}
\end{align}

The following lemma gives the complexity of this problem.
\begin{lemma}
\label{lemma:ComplexityOfFeasibilityOfBilevelProgramming}
The decision problems \DecisionProblemFeasibilityOfBilevelIntegerLP{} and \DecisionProblemFeasibilityOfBilevelNonNegativeIntegerLP{} are in \ClassSigmaP{2}.
\end{lemma}

\begin{proof}
For $\DomainPlaceholder \in \{\Integers, \NonNegativeIntegers\}$, the following \ClassNP{} machine $M$ with an \ClassNP{} oracle decides \DecisionProblemFeasibilityWithDomainPlaceholder{}:
\begin{enumerate}
\item
Nondeterministically guess a feasible solution $(x_1, x_2)$ of polynomial size.
\item \label{check1}
Verify that it satisfies
$
A_{1 1} x_1 + A_{1 2} x_2 \BILPConstraintSense b_1
$
and
$
A_{2 1} x_1 + A_{2 2} x_2 \BILPConstraintSense b_2.
$
\item \label{check2}
Verify that $x_2$ is optimal for the lower-level problem by querying the oracle.
\item
Accept if checks~\ref{check1} and~\ref{check2} pass.
\end{enumerate}
Proposition~\ref{prop:EstimateOfOptimalObjectiveValue} establishes a polynomial bound on the encoding size of a feasible solution in step~(i).
The input instance is a ``yes'' instance of \DecisionProblemFeasibilityWithDomainPlaceholder{} if and only if there exist nondeterministic choices under which $M$ accepts.
\end{proof}

\begin{proof}[Theorem~\ref{theorem:ComplexityOfBilevelIntegerLP}]
Hardness follows from Theorem~\ref{theorem:ComplexityOfBilevelBinaryLP}.
In light of Lemma~\ref{lemma:ComplexityOfFeasibilityOfBilevelProgramming}, membership follows from reducing \DecisionProblemBilevelLPWithDomainPlaceholder{} to \DecisionProblemFeasibilityWithDomainPlaceholder{}, $\DomainPlaceholder \in \{\Integers, \NonNegativeIntegers\}$: One can decide the former by checking the feasibility of the input instance with an additional upper-level constraint ``the objective value is less than or equal to the threshold''.
\end{proof}

\section{Computational Complexity Under Structural Restrictions}
\label{sec:FixedDimensions}

In this section, we consider the effects of structural assumptions.

\subsection{Fixed Total Number of Variables}

First, we consider a tractable case.

\begin{proposition}
\label{Prop:FixedTotalNumberOfVariables}
Let $k \in \NonNegativeIntegers$ and $\DomainPlaceholder \in \{ \Integers,\NonNegativeIntegers,\BinarySet \}$.
For any \BILPWithDomainPlaceholder{} instance satisfying $\NVariables_{1}+\NVariables_{2}=k$, an optimal solution to \eqref{Eq:BILP} can be computed, or infeasibility or unboundedness can be detected, in polynomial time.
\end{proposition}

\citet{koppe2010parametric} proved Proposition~\ref{Prop:FixedTotalNumberOfVariables} under the additional assumption that the feasible set is bounded. Their algorithm computes the optimal objective value and an optimal solution by binary search, relying on polynomial bounds on their encoding sizes.
Proposition~\ref{prop:EstimateOfOptimalObjectiveValue} provides a polynomial encoding size bound even without the boundedness of the feasible set. 
Hence, their algorithm extends to the general case with straightforward modifications.

\subsection{Fixed Number of Lower-Level Variables or Constraints}

Next, we consider cases where the number of lower-level variables or constraints is fixed.
% TODO bimodular matrix.
% We show that in these cases the problem goes down one level of the polynomial hierarchy and becomes \ClassNP{}-complete.
% We also establish the results where the lower-level problem has a bimodular constraint matrix, which makes the lower-level problem polynomially solvable given the upper-level decision.

\begin{proposition}
\label{Prop:FixedLowerLevel}
For any $k \in \NonNegativeIntegers$, the following statements hold.
\begin{enumerate}
\item
The decision problems \DecisionProblemBilevelIntegerLP{}, \DecisionProblemBilevelNonNegativeIntegerLP{}, and \DecisionProblemBilevelBinaryLP{} with the input restricted to satisfy $\NVariables_{2} = k$ are \ClassNP{}-complete.
Moreover, this remains valid even under the additional assumptions of conditions~\ref{Condition1}--\ref{Condition3}.
\item
The decision problem \DecisionProblemBilevelIntegerLP{} with the input restricted to satisfy $\NConstraints_{2} = k$ is \ClassNP{}-complete.
Moreover, this remains valid even under the additional assumptions of conditions~\ref{Condition1}--\ref{Condition3}.
\item
% \NagisaSideComment{\\It is not known if \DecisionProblemBilevelNonNegativeIntegerLP{} with $\NConstraints_{1} = 0, \NConstraints_{2} = 1$ is hard or not. But we do know \DecisionProblemBilevelBinaryLP{} with $\NConstraints_{1} = 0, \NConstraints_{2} = 1$ is \ClassSigmaP{2}-hard by our KNAPSACK argument below.}
The decision problems \DecisionProblemBilevelNonNegativeIntegerLP{} and \DecisionProblemBilevelBinaryLP{} with the input restricted to satisfy $\NConstraints_{2} = 1$ are \ClassSigmaP{2}-complete.
Moreover, this remains valid even under the additional assumptions of conditions~\ref{Condition1} and~\ref{Condition3}.
\item
The decision problems \DecisionProblemBilevelNonNegativeIntegerLP{} and \DecisionProblemBilevelBinaryLP{} with the input restricted to satisfy $\NConstraints_{2} = k$ and condition~\ref{Condition2} are \ClassNP{}-complete.
Moreover, this remains valid even under the additional assumptions of conditions~\ref{Condition1} and~\ref{Condition3}.
\end{enumerate}
\end{proposition}

% \NagisaSideComment{\\Note that the interdiction KNAPSACK does not prove (iii) as the interdiction requires many lower-level constraints.}
We note that the hardness result in assertion~(iii) under condition~\ref{Condition1} is implied by the hardness result of a bilevel knapsack problem studied by \citet{caprara2014study}.
We establish the result additionally imposing condition~\ref{Condition3} by showing that a minmax knapsack problem with a single shared capacity constraint is \ClassSigmaP{2}-hard.

\begin{proof}
\mbox{}
\begin{enumerate}
\item[(i), (ii), (iv)]
The hardness follows from that of the feasibility problem of single-level binary LP.
The membership follows from Proposition~\ref{prop:EstimateOfOptimalObjectiveValue} and the polynomial-time solvability of the lower-level problem: Nondeterministically guess a feasible solution $(x_1, x_2)$ of polynomial encoding size whose objective value is less than or equal to the given threshold, and verify the feasibility, including the optimality with respect to the lower-level problem, in polynomial time.
When $\DomainPlaceholder \in \{\Integers, \NonNegativeIntegers, \BinarySet\}$ and $\NVariables_2$ is fixed, \citet{lenstra1983integer} shows that the lower-level problem can be solved in polynomial time.
When $\DomainPlaceholder = \Integers$ and $\NConstraints_2$ is fixed, one can compute the Hermite normal form of the constraint matrix \cite{Schrijver1998} and reduce the problem to an integer program in a fixed dimension, again solvable with the method of \citet{lenstra1983integer}.
When $\DomainPlaceholder \in \{\NonNegativeIntegers, \BinarySet\}$ and $\NConstraints_2$ is fixed together with condition~\ref{Condition2}, the lower-level problem can be transformed to a 4-block integer LP instance and can be solved with the method of \citet{Lassota2026} in polynomial time.

\item[(iii)]
The membership follows from Theorem~\ref{theorem:ComplexityOfBilevelIntegerLP}.
We show the hardness by modifying the reduction from~\ref{Eq:3SAT} to KNAPSACK presented by \citet{krentel1986complexity}; this reduction transforms $\BILPQFreeBooleanFormula \in \BILPSetOfThreeCNF$ into
$$
\max_{x, \bar{x}, y} \{
\alpha^{\top} x
+
\bar{\alpha}^{\top} \bar{x}
+
\beta^{\top} y
:
\alpha^{\top} x
+
\bar{\alpha}^{\top} \bar{x}
+
\beta^{\top} y
\le W,
x, \bar{x} \in \BinarySet^{\BILPSATNVar}, y \in \BinarySet^{\BILPSATNVar'}
\},
$$
where $\BILPSATNVar', W, \alpha, \bar{\alpha}, \beta$ are positive integer (vectors), alongside an integer $M$ such that the following holds:
\begin{enumerate}
\item
For any binary vector $x'$, $(x', \OneVector - x', \ZeroVector)$ is feasible;
\item
Any feasible point $(x, \bar{x}, y)$ with an objective value greater than or equal to $M$ satisfies $x = \OneVector - \bar{x}$, and $x$ is a binary vector satisfying $\BILPQFreeBooleanFormula$;
\item
For any binary vector $x'$ satisfying $\BILPQFreeBooleanFormula$, there exists a feasible point $(x, \bar{x}, y)$ such that $x = x'$, $\bar{x} =  \OneVector - x'$, and its objective value is greater than or equal to $M$.
\end{enumerate}

A standard trick for KNAPSACK (for example, see Problem 9.5.33 in \citet{papadimitriou1993computational}) provides a modification of the above construction such that we can 1) relax the binary domain to nonnegative integers and 2) restrict $W = M$ without altering the solutions.
Thus, we can assume that
$$
\max_{x, \bar{x}, y} \{
\alpha^{\top} x
+
\bar{\alpha}^{\top} \bar{x}
+
\beta^{\top} y
:
\alpha^{\top} x
+
\bar{\alpha}^{\top} \bar{x}
+
\beta^{\top} y
\le W,
x, \bar{x} \in \DomainPlaceholder^{\BILPSATNVar}, y \in \DomainPlaceholder^{\BILPSATNVar'}
\}
$$
is such that, for $\DomainPlaceholder \in \{\BinarySet, \NonNegativeIntegers\}$, any feasible solution with an objective value greater than or equal to $W$ is binary, and the properties~(a)--(c) above hold.

Now consider $\BILPBooleanFormula \in \BILPSetOfBooleanFormulae$ of form~\eqref{Eq:QSATH}, i.e., $\BILPBooleanFormula := \exists s \forall t [\BILPQFreeBooleanFormula(s, t) = 0]$ for some $\BILPQFreeBooleanFormula \in \BILPSetOfThreeCNF$.
Following the same argument, we obtain
$$
\min_{s \in \BinarySet^{\BILPSATNVar}}
\left\{
\alpha^{\top} s +
\max_{\bar{s}, t, \bar{t}, y}
\left\{
\bar{\alpha}^{\top} \bar{s} + \gamma^{\top} t + \bar{\gamma}^{\top} \bar{t} + \beta^{\top} y
:
\begin{array}{l}
\alpha^{\top} s + \bar{\alpha}^{\top} \bar{s} + \gamma^{\top} t + \bar{\gamma}^{\top} \bar{t} + \beta^{\top} y \le W,
\\
\bar{s}, t, \bar{t} \in \DomainPlaceholder^{\BILPSATNVar}, y \in \DomainPlaceholder^{\BILPSATNVar'}
\end{array}
\right\}
\right\},
$$
where the inner maximization represents a KNAPSACK instance.
The optimal objective value of this minmax problem is less than or equal to $W - 1$ if and only if $\BILPBooleanFormula$ is true.
Note that in case of $\DomainPlaceholder = \NonNegativeIntegers$, $s \in \BinarySet^{\BILPSATNVar}$ can be modeled as $s \in \NonNegativeIntegers^{\BILPSATNVar}$ together with non-linking upper-level constraints $s \le \OneVector{}$.
This establishes the \ClassSigmaP{2}-hardness.

\end{enumerate}
\end{proof}

More generally, if we impose an efficiently verifiable structural assumption that guarantees polynomial-time solvability of the lower-level problem, the complexity of the resulting decision problem is in \ClassNP{}. 
For example, when the lower-level constraint matrix $A_{22}$ is restricted to be totally unimodular, the decision problems \DecisionProblemBilevelIntegerLP{}, \DecisionProblemBilevelNonNegativeIntegerLP{}, and \DecisionProblemBilevelBinaryLP{} are \ClassNP{}-complete. 
Note that total unimodularity can be recognized in polynomial time~\cite{Schrijver1998}.
% If $A_{2 2}$ has full column rank and is totally bimodular, the lower-level problem is tractable, but the recognition of such problem is open \citet{stephan2017}.
% see https://arxiv.org/pdf/2106.14980

\subsection{Fixed Number of Upper-Level Variables and Constraints}

In this section, we study cases in which the number of upper-level variables is fixed and there are no upper-level constraints.
Note that when there are no upper-level constraints, condition~\ref{Condition1} holds vacuously.

\begin{proposition}
For any $k \in \PositiveIntegers$, the decision problems \DecisionProblemBilevelIntegerLP{} and \DecisionProblemBilevelNonNegativeIntegerLP{} with the input restricted to satisfy $\NVariables_{1} = k$ and $\NConstraints_{1} = 0$ are \ClassSigmaP{2}-complete.
Moreover, this remains valid even under the additional assumptions of conditions~\ref{Condition2} and~\ref{Condition3}.
\end{proposition}

\begin{proof}
Membership in \ClassSigmaP{2} follows from Theorem~\ref{theorem:ComplexityOfBilevelIntegerLP}.
Thus, we only need to establish hardness.
For $\DomainPlaceholder \in \{\Integers, \NonNegativeIntegers\}$ and $\BILPBooleanFormula \in \BILPSetOfBooleanFormulae$ of form~\eqref{Eq:QSATH}, consider the following instance:
\begin{align}
\label{Eq:JeroslowWithSingleUpperLevelVariables}
\min_{\theta \in \DomainPlaceholder}
\max_{s, t, v}
\left\{
v :
\begin{array}{l}
A_{\BILPQFreeBooleanFormula}
\begin{pmatrix}
s \\ t
\end{pmatrix}
\ge
v \OneVector
- b_{\BILPQFreeBooleanFormula},
\\
2^{\BILPSATNVar - 1} s_{1}
+
2^{\BILPSATNVar - 2} s_{2}
+
\cdots
+
s_{\BILPSATNVar}
=
\theta,
\\
t\leq \OneVector, v\leq 1, s\leq \OneVector, t \in \DomainPlaceholder^{\BILPSATNVar}, v \in \DomainPlaceholder, s \in \DomainPlaceholder^{\BILPSATNVar}
\end{array}
\right\}.
\end{align}
We have $\OptValFunc\eqref{Eq:JeroslowWithSingleUpperLevelVariables} = \OptValFunc\eqref{Eq:Jeroslow}$, and one can obtain an equivalent instance satisfying condition~\ref{Condition2} by introducing a polynomial number of auxiliary variables as in~\eqref{SAT:entresin01}.
\end{proof}

Next, we consider binary instances.
For each $l \in \PositiveIntegers{}$, we denote by $\ClassBH{l}$ the $l$-th level of the Boolean hierarchy~\cite{cai1988boolean}.
It is defined as follows:
\begin{alignat*}{2}
\ClassBH{1} &= \ClassNP{},
\\
\ClassBH{l} &= \{ L_1 \cap \overline{L_2} : L_1 \in \ClassBH{l - 1}, L_2 \in \ClassNP{}\}, && \ \ \forall l = 2, 4, \ldots,
\\
\ClassBH{l} &= \{ L_1 \cup L_2 : L_1 \in \ClassBH{l - 1}, L_2 \in \ClassNP{}\}, && \ \ \forall l = 3, 5, \ldots.
\end{alignat*}
In particular, $\ClassBH{2} = \ClassDP{}$, the complexity class introduced by~\citet{papadimitriou1982complexity}.
% THEOREM 2.1.1
% There exist various characterizations of the Boolean hierarchy.
% In particular, 
For each $l \ge 1$, we have~\cite{cai1988boolean}
\begin{align}
\label{Eq:BHAsUnionOfSetDifferences}
\ClassBH{2 l}
&=
\left\{
\bigcup_{i = 1, \ldots, l}
(L_{2 i - 1} \cap \overline{L_{2 i}})
:
L_i \in \ClassNP{}, i = 1, \ldots, 2 l
\right\}.
\end{align}

\begin{proposition}
\label{prop:BILPFixedUpperLevelVariables}
\mbox{}
For any $k \in \NonNegativeIntegers$, the following statements hold:
\begin{enumerate}
\item
% \NagisaSideComment{\\The claim holds even with $n_1 = 0$.}
The decision problem \DecisionProblemBilevelBinaryLP{} with the input restricted to satisfy $\NVariables_{1} = k$ and $\NConstraints_{1} = 0$ is \ClassDeltaP{2}-complete;
\item
The decision problem \DecisionProblemBilevelBinaryLP{} with the input restricted to satisfy $\NVariables_{1} = k$, $\NConstraints_{1} = 0$, and condition~\ref{Condition2} is \ClassThetaP{2}-complete;
\item
The decision problem \DecisionProblemBilevelBinaryLP{} with the input restricted to satisfy $\NVariables_{1} = k$, $\NConstraints_{1} = 0$, and condition~\ref{Condition3} is \ClassBH{2^{k+1}}-complete.
Moreover, this remains valid even under the additional assumption of condition~\ref{Condition2}.
\end{enumerate}
\end{proposition}

\begin{proof}
\begin{enumerate}
\item
\textbf{Membership.}
We demonstrate the existence of a polynomial-time Turing machine $M$ with an \ClassNP{} oracle that computes the optimal objective value (or reports infeasibility or unboundedness), from which the result follows.
Note that given $\NConstraints_1 = 0$,
$
\OptValFunc\eqref{Eq:BILP} =
\min_{x_1 \in \BinarySet^{k}}
\{
\phi(x_1)
\},
$
where
\begin{equation*}
\phi(x_1)
:=
c_{1 1}^{\top} x_1 +
\min_{x_2}
\left\{
c_{1 2}^{\top} x_2
:
x_2
\in
\Argmin_{x_2'}
\left\{
c_{2 2}^{\top} x_2'
:
\begin{array}{l}
A_{2 1} x_1 + A_{2 2} x_2' \ge b_2,
\\
x_2' \in \BinarySet^{\NVariables_2}
\end{array}
\right\}
\right\}.
\end{equation*}
The machine $M$ evaluates $\phi(x_1)$ for each $x_1 \in \{0, 1\}^{k}$ and outputs the minimum value among them.
For each $x_1 \in \{0, 1\}^k$, $M$ queries whether
$$
\min_{x_2}
\{
c_{2 2}^{\top} x_2 :
A_{2 1} x_1 + A_{2 2} x_2
\BILPConstraintSense
b_2,
x_2 \in \BinarySet^{\NVariables_2}
\}
$$
is feasible.
If not, the current choice of $x_1$ is infeasible. %, i.e., $\phi(x_1) = \infty$.
Otherwise, using binary search, $M$ computes its optimal objective value $v'$.
Now, observe
$$
\phi(x_1)
= c_{1 1}^{\top} x_1 +
\min_{x_2}
\{
c_{1 2}^{\top} x_2
:
c_{2 2}^{\top} x_2 \le v',
A_{2 1} x_1 + A_{2 2} x_2
\BILPConstraintSense
b_2,
x_2 \in \BinarySet^{\NVariables_2}
\},
$$
which can also be computed via binary search.
Each evaluation of $\phi(x_1)$ runs in polynomial time, which implies the claim.

\noindent
\textbf{Hardness.}
We show hardness via a reduction from the following problem:
\begin{equation}
\tag{LEXSAT}
\label{Eq:LEXSAT}
\parbox{0.8\textwidth}{%
Given $\BILPQFreeBooleanFormula \in \BILPSetOfThreeCNF$, is $\BILPQFreeBooleanFormula$ satisfiable, and does the lexicographically maximum satisfying assignment have least significant bit equal to $1$?
}
\end{equation}
This is \ClassDeltaP{2}-complete~\cite{krentel1986complexity}.
Let $\BILPQFreeBooleanFormula \in \BILPSetOfThreeCNF$ with $\BILPSATNVar$ Boolean variables and consider
% \NagisaSideComment{This does not have any upper-level variables or constraints.}
\begin{equation}
\label{Eq:LEXSATBilevelFormulation}
\min_{u}
\left\{
-u_{\BILPSATNVar}
:
u
\in \Argmax_{u'}
\left\{
\sum_{i = 1}^{\BILPSATNVar} 2^{\BILPSATNVar - i} u_{i}'
:
A_{\BILPQFreeBooleanFormula} u' \ge \OneVector - b_{\BILPQFreeBooleanFormula},
u' \in \BinarySet^{\BILPSATNVar}
\right\}
\right\},
\end{equation}
which evaluates to $-1$ or less if and only if $\BILPQFreeBooleanFormula$ is a ``yes'' instance for~\ref{Eq:LEXSAT}.

\item
\textbf{Membership.}
It follows since the machine described in the proof of assertion~(i) makes a logarithmic number of queries under condition~\ref{Condition2}.

\noindent
\textbf{Hardness.}
We show hardness via a reduction from the following problem:
\begin{equation}
\tag{MODSAT}
\label{Eq:KMult}
\parbox{0.8\textwidth}{%
Given $\BILPQFreeBooleanFormula \in \BILPSetOfThreeCNF$ and an integer $\DecisionVersionThreshold$, is the maximum number of simultaneously satisfiable clauses in $\BILPQFreeBooleanFormula$ divisible by $\DecisionVersionThreshold$?
}
\end{equation}
This is \ClassThetaP{2}-complete~\cite{krentel1986complexity}.
Let $\BILPQFreeBooleanFormula \in \BILPSetOfThreeCNF$ with $\BILPSATNVar$ Boolean variables and $\BILPSATNClause$ clauses, and let $\DecisionVersionThreshold \in \NonNegativeIntegers$ be an integer between $2$ and $\BILPSATNClause$. 
Consider
\begin{equation}
\label{eq:ModSatBilevelFormulation}
\min_{u, v, \alpha, \beta}
\left\{
\sum_{i = 1}^{\DecisionVersionThreshold - 1} \beta_{i}
:
\begin{pmatrix}
u \\ v \\ \alpha \\ \beta
\end{pmatrix}
\in \Argmax_{u', v', \alpha', \beta'}
\left\{
\sum_{i = 1}^{\BILPSATNClause} v_{i}'
:
\begin{array}{ll}
A_{\BILPQFreeBooleanFormula} u' \ge v' - b_{\BILPQFreeBooleanFormula},
\\
\displaystyle
\DecisionVersionThreshold \alpha'
+
\sum_{i = 1}^{\DecisionVersionThreshold - 1} \beta_i'
=
\sum_{i = 1}^{\BILPSATNClause} v_i',
\\
\alpha' \le \lceil \BILPSATNClause / \DecisionVersionThreshold \rceil,
\\
u' \in \BinarySet^{\BILPSATNVar}, v' \in \BinarySet^{\BILPSATNClause}, \alpha' \in \NonNegativeIntegers, \beta' \in \BinarySet^{\DecisionVersionThreshold - 1}
\end{array}
\right\}
\right\}.
\end{equation}
Let $l$ be the maximum number of simultaneously satisfiable clauses in $\BILPQFreeBooleanFormula$.
Any optimal solution to the lower-level problem satisfies $\sum_{i = 1}^{\BILPSATNClause} v_{i}' = l$ and $\sum_{i = 1}^{\DecisionVersionThreshold - 1} \beta_i' \equiv l \pmod \DecisionVersionThreshold$.
Thus, $\OptValFunc\eqref{eq:ModSatBilevelFormulation} \le 0$ if and only if $(\BILPQFreeBooleanFormula, \DecisionVersionThreshold)$ is a ``yes'' instance for~\ref{Eq:KMult}.
Note that the resulting instance contains non-binary variables as well as coefficients larger than $1$.
However, since $\DecisionVersionThreshold$ is at most $\BILPSATNClause$, one can obtain an equivalent instance satisfying condition~\ref{Condition2} by introducing a polynomial number of auxiliary binary variables.

\item
\textbf{Membership.}
With~\ref{Condition3} and $\NConstraints_1 = 0$, 
% the problem in \eqref{Eq:BILP} can be rewritten as
$
\OptValFunc\eqref{Eq:BILP} =
\min_{x_1 \in \BinarySet^{k}}
\left\{
\phi'(x_1)
\right\},
$
where
$$
\phi'(x_1) :=
c_{1 1}^{\top} x_1
+
\max_{x_2}
\left\{
c_{1 2}^{\top} x_2
:
A_{2 1} x_1 + A_{2 2} x_2 \ge b_2,
\
x_2 \in \BinarySet^{\NVariables_2}
\right\}
$$
if the lower-level problem is feasible, and $\phi'(x_1) = \infty$ otherwise.
% We note that $\min_{x_1 \in \BinarySet^{k}} \left\{ \phi'(x_1) \right\} \le \DecisionVersionThreshold$ if and only if
The following equivalent characterization of the condition $\min_{x_1 \in \BinarySet^{k}} \left\{ \phi'(x_1) \right\} \le \DecisionVersionThreshold$ implies the membership:
\begin{align*}
&
\underbrace{\bigvee_{x_1 \in \BinarySet^{k}}}_{\substack{\text{There exists}\\\text{$x_1$ such that}}}
\left(
\underbrace{
\left(
\max_{x_2 \in \BinarySet^{\NVariables_2}}
\left\{
0
:
A_{2 1} x_1 + A_{2 2} x_2 \ge b_2
\right\}
\ge
0
\right)
}_{\text{the lower-level problem is feasible}}
\right.
\\
&
\qquad\qquad
\bigwedge
\left.
\underbrace{
\neg
\left(
\max_{x_2 \in \BinarySet^{\NVariables_2}}
\left\{
c_{1 2}^{\top} x_2
:
A_{2 1} x_1 + A_{2 2} x_2 \ge b_2
\right\}
\ge
\DecisionVersionThreshold + 1 - c_{1 1}^{\top} x_1
\right)
}_{\text{the lower-level problem is infeasible or $\phi'(x_1) \le \DecisionVersionThreshold$}}
\right).
\end{align*}
% Together with~\eqref{Eq:BHAsUnionOfSetDifferences}, this establishes membership in \ClassBH{2^{k + 1}}.

\noindent
\textbf{Hardness.}
We consider a reduction from the following decision problem:
\begin{equation}
\notag
\parbox{0.8\textwidth}{Given $\BILPQFreeBooleanFormula_1, \ldots, \BILPQFreeBooleanFormula_{2^{k + 1}} \in \BILPSetOfThreeCNF$, is the number of satisfiable formulae odd?}
\end{equation}
This is known to be \ClassBH{2^{k + 1}}-complete~\cite{kobler1987difference}.
For $i = 1, \ldots, 2^{k + 1}$, let $\BILPQFreeBooleanFormula_i \in \BILPSetOfThreeCNF$ with $\BILPSATNVar_i$ Boolean variables and consider
\begin{align}
\label{Eq:OddKSATBilevel}
\min_{p \in \BinarySet^{k}}
\left\{
\max_{q, u, v}
\left\{
q
:
\begin{array}{l}
A_{\BILPQFreeBooleanFormula_i} u_{i} \ge v_{i} \OneVector - b_{\BILPQFreeBooleanFormula_i},
\\
\displaystyle
q + 1 + \sum_{i = 1}^{k} 2^{i} p_{i} \le \sum_{i = 1}^{2^{k + 1}} v_i,
\\
q \in \BinarySet, u_i \in \BinarySet^{\BILPSATNVar_i}, i = 1, \ldots, 2^{k + 1}, v \in \BinarySet^{2^{k + 1}}
\end{array}
\right\}
\right\}.
\end{align}
Let $l$ denote the number of satisfiable formulae. 
The lower-level problem is infeasible if $1 + \sum_{i=1}^k 2^i p_i > l$, has optimal value $0$ if $1 + \sum_{i=1}^k 2^i p_i = l$, and optimal value $1$ otherwise.
Hence, $\OptValFunc\eqref{Eq:OddKSATBilevel} \le 0$ if and only if $l$ is odd. 
The resulting instance can be reformulated to satisfy condition~\ref{Condition2}.
\end{enumerate}
\end{proof}

The assumption $m_1 = 0$ in assertions~(i) and~(ii) can be relaxed with only minor modifications to the proofs (see \cite{koppe2010parametric}). 
In contrast, the assumption is essential in~(iii): \DecisionProblemBilevelBinaryLP{} with $n_1 = 0, m_1 = 1$ and condition~\ref{Condition3} is \ClassDeltaP{2}-complete, as a modification of \eqref{Eq:LEXSATBilevelFormulation} into a minmax problem with a single linking constraint $u_r \le -1$ yields an instance that is feasible if and only if the \ref{Eq:LEXSAT} instance is a ``yes'' instance.

When $\NVariables_1 = 0$, the proofs for Proposition~\ref{prop:BILPFixedUpperLevelVariables} extend to bilevel integer LP easily.
% The proof is a straightforward adaptation of that of Proposition~\ref{prop:BILPFixedUpperLevelVariables} and is hence omitted.
% (the hardness result in assertion~(i) follows from reformulating~\eqref{Eq:LEXSATBilevelFormulation} as a \BILPInteger\textbf{}s{} or \BILPNonNegativeIntegers{} instance with condition~\ref{Condition2} using auxiliary variables).

\begin{proposition}
\label{prop:BilevelIntegerLPWithoutUpperLevelVariables}
The following statements hold:
\begin{enumerate}
\item
The decision problems \DecisionProblemBilevelIntegerLP{} and \DecisionProblemBilevelNonNegativeIntegerLP{} with the input restricted to satisfy $\NVariables_{1} = 0$ are \ClassDeltaP{2}-complete.
Moreover, this remains valid even under the additional assumptions of condition~\ref{Condition2} and $\NConstraints_{1} = 0$.
\item
The decision problems \DecisionProblemBilevelIntegerLP{} and \DecisionProblemBilevelNonNegativeIntegerLP{} with the input restricted to satisfy $\NVariables_{1} = 0$, $\NConstraints_{1} = 0$, and condition~\ref{Condition3} are \ClassDP{}-complete.
Moreover, this remains valid even under the additional assumption of condition~\ref{Condition2}.
\end{enumerate}
\end{proposition}

\section{Extension to Pricing Problems}
\label{sec:Pricing}

In this section, we study the pricing variant of bilevel integer LP.
This problem models a setting wherein the upper-level player sets prices for a set of items, and the lower-level player chooses items to maximize their utility minus the price paid.

\subsection{Problem Definition}

Let $\DomainPlaceholder \in \{\Integers, \NonNegativeIntegers, \BinarySet\}$ and consider
% The pricing problem seeks to find prices $x_1$ that minimize the negative revenue of the upper-level player:
\begin{align}
\min_{x_1, x_2}
\left\{
\begin{array}{l}
c_{1 1}^{\top} x_1 + c_{1 2}^{\top} x_2 + x_2^{\top} D_{1} x_1
\\[0.8em]
\qquad\ \
:
\begin{array}{l}
A_{1 1} x_1 + A_{1 2} x_2
\ge
b_1,
x_1 \in \DomainPlaceholder^{n_1},
\\
\displaystyle
x_2
\in
\Argmin_{x_2'}
\left\{
c_{2 2}^{\top} x_2' + x_1^{\top} D_{2} x_2'
:
\begin{array}{l}
A_{2 1} x_1 + A_{2 2} x_2'
\ge
b_2,
\\
x_2' \in \DomainPlaceholder^{n_2}
\end{array}
\right\}
\end{array}
\end{array}
\right\}.
\label{Eq:Pricing-BILP}
\end{align}

We write $D = ( D_1, D_2 )$.
When $A_{2 1} = \ZeroMatrix$, the problem reduces to a typical pricing problem, where the feasible set of the lower-level problem is independent of the upper-level decisions.
For $\DomainPlaceholder \in \{\Integers, \NonNegativeIntegers, \BinarySet\}$, we consider the decision problem:
\begin{align}
&
\parbox{0.8\textwidth}{%
Given $(A, b, c, D)$ and an integer $\DecisionVersionThreshold$, does there exist a feasible solution for \eqref{Eq:Pricing-BILP} whose objective value is less than or equal to $\DecisionVersionThreshold$?
}
\tag{\DecisionProblemPricingBilevelLPWithDomainPlaceholder}
\label{Eq:Pricing-BILP-VAL}
\end{align}

The modified conditions are as follows:
\begin{enumerate}[label=(C\arabic*$'$)]
\item
\label{Condition1Prime}
There are no linking constraints, i.e., $A_{1 2} = \ZeroMatrix$;
\item
\label{Condition2Prime}
All entries in $A$, $b$, $c$, $D$, and $\DecisionVersionThreshold$ are integers between $-1$ and $1$;
\item
\label{Condition3Prime}
It is a minmax problem, i.e., $c_{1 2} = -c_{2 2}$ and $D_1 = -D_2$;
\item
\label{Condition4Prime}
The feasible set of the lower-level problem is independent of the upper-level decision, i.e., $A_{2 1} = \ZeroMatrix$.
\end{enumerate}

We establish the following results.

\begin{theorem}
\label{theorem:ComplexityOfPricingBilevelProgramming}
The decision problems 
\DecisionProblemPricingBilevelIntegerLP{}, \DecisionProblemPricingBilevelNonNegativeIntegerLP{}, and \DecisionProblemPricingBilevelBinaryLP{} are \ClassSigmaP{2}-complete.
Moreover, this remains valid even when the input is restricted to satisfy conditions~\ref{Condition1Prime}--\ref{Condition4Prime}.
\end{theorem}

The proof is analogous to that of Theorem~\ref{theorem:ComplexityOfBilevelIntegerLP}; thus, we only provide a sketch.

\textbf{Membership.}
We first establish a polynomial bound on the encoding size of an optimal solution, which is parallel to Proposition~\ref{prop:EstimateOfOptimalObjectiveValue}.
This can be achieved using the universal test set for integer programming~\cite{weismantel1998test}, which generalizes the test set in Lemma~\ref{lemma:TestSetOfILP}.
For any integer programming instance with a parametric cost and constraint right-hand side, the universal test set is a finite set of integer vectors such that any non-optimal solution can be improved by adding a vector from this set.
This allows us to rewrite the pricing bilevel integer LP instance as a single-level quadratic programming (QP) instance whose feasible set is a union of integer points within polyhedra.
By invoking the results of~\citet{pia2017mixed}, which show that integer QP is in \ClassNP{}, we can establish the existence of an optimal solution of polynomial size.
The remaining steps follow straightforwardly.

\textbf{Hardness.}
Given $\BILPBooleanFormula \in \BILPSetOfBooleanFormulae$ of the form \eqref{Eq:QSATH}, consider the following instance:
\begin{align}
\label{Eq:HardPricing}
&
\min_{s \in \BinarySet^{\BILPSATNVar}}
\max_{\bar{s}, t, v}
\left\{
-2 \sum_{i = 1}^{\BILPSATNVar} (s_i - \bar{s}_i)^2 + v 
:
\begin{array}{l}
A_{\BILPQFreeBooleanFormula} 
\begin{pmatrix}
\bar{s} \\ t
\end{pmatrix}
\ge
v \OneVector
- b_{\BILPQFreeBooleanFormula},
\\
\bar{s}, t \in \BinarySet^{\BILPSATNVar}, v \in \BinarySet
\end{array}
\right\}
\\
\notag
&
= 
\min_{s \in \BinarySet^{\BILPSATNVar}}
\left\{
-2 \sum_{i = 1}^{\BILPSATNVar} s_i
+
\max_{\bar{s}, t, v}
\left\{
\sum_{i = 1}^{\BILPSATNVar} (4 s_i \bar{s}_i - 2 \bar{s}_i) + v 
:
\begin{array}{l}
A_{\BILPQFreeBooleanFormula} 
\begin{pmatrix}
\bar{s} \\ t
\end{pmatrix}
\ge
v \OneVector
- b_{\BILPQFreeBooleanFormula},
\\
\bar{s}, t \in \BinarySet^{\BILPSATNVar}, v \in \BinarySet
\end{array}
\right\}
\right\}.
\end{align}
For any $s \in \BinarySet^{\BILPSATNVar}$, the optimal choice for the lower-level problem is $\bar{s} = s$.
Thus, $\OptValFunc\eqref{Eq:HardPricing} = 0$ if $\BILPBooleanFormula$ is true, and $1$ otherwise.
Furthermore, one can introduce extra variables to ensure condition~\ref{Condition2Prime}.
This establishes the \ClassSigmaP{2}-hardness.

\section*{Acknowledgements} 

We are grateful to Sergei Ketkov and Koen Lighthart for valuable discussions and for pointing us to relevant references.

\printbibliography

\end{document}